\documentclass{amsart}

\usepackage{amsmath}
\usepackage{amsfonts}
\usepackage{amsthm}
\usepackage{amssymb}
\usepackage{graphicx}
\usepackage[cmtip,all]{xy}
\usepackage{hyperref}
\usepackage[english]{babel}

\newtheorem{theorem}{Theorem}[section]
\newtheorem{lemma}[theorem]{Lemma}
\newtheorem{corollary}[theorem]{Corollary}
\newtheorem{proposition}[theorem]{Proposition}

\theoremstyle{definition}

\theoremstyle{remark}
\newtheorem{rem}[theorem]{Remark}

\numberwithin{equation}{section}

\DeclareMathOperator{\aut}{Aut}
\DeclareMathOperator{\fix}{Fix}
\DeclareMathOperator{\irr}{Irr}

\newcommand{\bb}[1]{\mathbb{#1}}

\newcommand\Z{{\mathbb{Z}}}

\newcommand\F{{\mathbb{F}}}
\renewcommand\P{{\mathbb{P}}}

\usepackage{marvosym}

\begin{document}

\title{A decomposition of the Jacobian of a Humbert-Edge curve}


\author{Robert Auffarth}
\address{R.~Auffarth \\Departamento de Matem\'aticas, Facultad de
Ciencias, Uni\-ver\-si\-dad de Chile\\ Las Palmeras 3425, 7800024 \~Nu\~noa, Santiago, Chile}
\email{rfauffar@uchile.cl}
\thanks{R.~Auffarth was partially supported by Conicyt via Fondecyt Grant 11180965}

\author{Giancarlo Lucchini Arteche}
\address{G.~Lucchini Arteche \\Departamento de Matem\'aticas, Facultad de
Ciencias, Universidad de Chile\\ Las Palmeras 3425, 7800024 \~Nu\~noa, Santiago, Chile}\curraddr{}
\email{luco@uchile.cl}
\thanks{G.~Lucchini Arteche was partially supported by Conicyt via Fondecyt Grant 11170016 and PAI Grant 79170034}

\author{Anita M. Rojas}
\address{A.~M.~Rojas \\Departamento de Matem\'aticas, Facultad de
Ciencias, Universidad de Chile\\ Las Palmeras 3425, 7800024 \~Nu\~noa, Santiago, Chile}
\email{anirojas@uchile.cl}
\thanks{A.~M.~Rojas was partially supported by Conicyt via Fondecyt Grant 1180073}

\subjclass[2010]{Primary 14H40. Secondary 14H37, 14K02}

\date{}

\begin{abstract}
A \textit{Humbert-Edge curve of type} $n\geq 2$ is a non-degenerate smooth complete intersection of $n-1$ diagonal quadrics in $\bb P^n$. Such a curve has an interesting geometry since it has a natural action of the group $(\mathbb{Z}/2\Z)^n$. We present here a decomposition of its Jacobian variety as a product of Prym-Tyurin varieties, and we compute the kernel of the corresponding isogeny.
\end{abstract}

\maketitle

\hfill\textit{To Herbert Lange on his 75th birthday.}

\section{Introduction}

An irreducible non-singular curve $X_n\subseteq\P^n$ over an algebraically closed field of characteristic different from 2 is a \textit{Humbert-Edge curve of type} $n\geq 2$ if it is non-degenerate (i.e. not contained in a hyperplane) and is the complete intersection of $n-1$ diagonal quadrics. Multiplication by $-1$ on the $i$th coordinate of $\P^n$ restricts to an automorphism $\sigma_i$ of $X_n$, and the group generated by these involutions is a subgroup $E_n$ of $\aut(X_n)$ isomorphic to $(\mathbb{Z}/2\Z)^n$. 

Humbert-Edge curves of type $4$ were first studied by Humbert \cite{Humbert} and later generalized to arbitrary type by Edge in \cite{Edge}. They were then studied more in depth in \cite{Rubi} from a complex-analytic point of view and in \cite{Alexis} from an algebraic point of view. In particular in this last article a coarse decomposition of the Jacobian variety of a Humbert-Edge curve was given. The purpose of this paper is to give a more refined decomposition of the Jacobian variety of a Humbert-Edge curve by using the automorphism group $E_n$. Our results imply the following (see Theorems \ref{main thm} and \ref{T:PT-K} for the specific statements):

\begin{theorem}\label{T:main summary}
The Jacobian of a Humbert-Edge curve of type $n$ splits isogenously as the product of lower-dimensional abelian subvarieties where the dimension of the largest factor grows logarithmically with respect to the genus of the curve. Moreover, each of these factors is a Prym-Tyurin variety of exponent $2^{n-3}$ for the curve.
\end{theorem}

We recall that a \textit{Prym-Tyurin variety of exponent $e$ for a curve $X$} is a principally polarized abelian variety $(A,\Theta)$ such that there exists an embedding $i:A\hookrightarrow JX$ that satisfies $i^*\Theta_X\equiv e\Theta$, where $\Theta_X$ is the theta divisor of $JX$ and $\equiv$ stands for numerical equivalence.

In a sense, our result means that the Jacobian of a Humbert-Edge curve has a fine decomposition. Isogeny decompositions of Jacobians of curves have been a subject of interest ever since Ekedahl and Serre \cite{ES} found examples of curves of many genera, the largest being 1297, whose Jacobian variety splits isogenously as the product of elliptic curves. They asked if the genera of completely decomposable Jacobians are bounded, and this is a question that remains open in characteristic 0.

We obtain here that the Jacobian variety of a Humbert-Edge curve $X_n$ is decomposed as a product of subvarieties which are Prym-Tyurin varieties for $X_n$, and are actually isogenous to Jacobian varieties. We were previously unaware of any examples that showed this type of behavior.

\begin{rem}
Note that Humbert-Edge curves correspond to generalized Fermat curves of type $(2,n)$. When submitting this article, we found out that in \cite[\S5]{quat}, a decomposition of the Jacobian of a generalized Fermat curve $X$ of type $(p,n)$ for $p$ prime is given in terms of pullbacks of Jacobians of quotients $X/H$ for subgroups $H\leq \aut(X)\simeq (\Z/p\Z)^n$ of index $p$ (i.e.~hyperplanes of the $\F_p$-vector space). In the case $p=2$, this decomposition turns out to correspond to the one we obtain here. However, our approach is different: we follow an inductive procedure instead of cutting out by hyperplanes. Our point of view allows us to describe the decomposition more precisely: we give the exact number of factors and their corresponding dimensions. Moreover, we compute the kernel of the decomposition and the induced polarization on each factor, finding that all of them are Prym-Tyurin varieties. We intend to generalize our results to arbitrary generalized Fermat curves in the future.
\end{rem}

\textsc{Acknowledgements.} We were inspired to work on the Jacobian decomposition of Humbert-Edge curves by a talk given by Alexis Zamora at the Workshop \textit{Geometry at the Frontier III}, which took place in Puc\'on, Chile, November 12-16, 2018. We would like to thank the organizers for a wonderful conference. We would also like to thank the anonymous referee for his or her useful comments.

\section{Main theorem}

By \cite[Lemma 3.2]{Alexis}, a Humbert-Edge curve $X_n$ is of genus $2^{n-2}(n-3)+1$ and the quotient of $X_n$ by each $\sigma_i$ is a Humbert-Edge curve of type $n-1$. In particular, if we denote by $S_n$ the set $\{\sigma_0,\ldots,\sigma_n\}\subset\aut(X_n)$, then for each subset $T\subset S_n$ we can consider the curve $X_{T}:=X_n/\langle T\rangle$. It is a Humbert-Edge curve of type $n-|T|$ whose corresponding subgroup of automorphisms $E_T\leq\aut(X_T)$ is naturally isomorphic to $E_n/\langle T\rangle$ and generated by the images of the $\sigma_i\in S_n\smallsetminus T$, as can easily be proved by induction.\\

By \cite[Proposition 3.5]{Alexis}, the Jacobian of $X_n$ has the following de\-com\-po\-si\-tion:
\begin{equation}\label{eq decomp alexis}
JX_n\sim A\oplus JX_n^-,
\end{equation}
where $A=\sum_{i=0}^n\pi_i^* J(X/\sigma_i)$ (here $\pi_i:X_n\to X_n/\sigma_i$ is the natural projection) and $JX_n^-$ is the neutral connected component of
\[\{x\in JX_n:\sigma_i(x)=-x\text{ for all }i\}.\]
Note that in \cite{Alexis} the connectedness of this set is assumed, yet not proved.\\

Our goal is to push this decomposition further by using the previous remark on the quotients $X_T$ for $T\subset S_n$. Our main result states the following:

\begin{theorem}\label{main thm}
Let $X_n$ be a Humbert-Edge curve of type $n\geq 3$. Then we have the following decomposition of $JX_n$:
\[JX_n\sim\underset{|T|\leq n-3}{\bigoplus_{T\subset S_n}}\pi_T^*JX_{T}^{-},\]
where $\pi_T:X_n\to X_T$ is the natural projection. Moreover, if $n-|T|$ is even, then $JX_T^{-}$ is trivial; and if $n-|T|=2m+1$, then $\dim JX_{T}^{-}=m$. In particular, for $1\leq m\leq \lfloor\frac{n-1}{2}\rfloor$, there are exactly $\binom{n+1}{2m+2}$ summands of dimension $m$.
\end{theorem}

We observe that in particular, we can always find $\binom{n+1}{4}$ elliptic curves in the decomposition and the largest factor above is of dimension $ \lfloor\frac{n-1}{2}\rfloor$. What is nice about this is that, while the genus of $X_n$ grows exponentially in $n$, the dimension of the largest factor in its Jacobian decomposition grows linearly and is therefore relatively small.

\section{Proof of the main theorem}
Note that the last statement of Theorem \ref{main thm} follows immediately from com\-bi\-na\-to\-rial considerations. Let us start then by analyzing the dimension of $JX_n^{-}$.

\begin{lemma}\label{Anitaslemma}
Let $X_n$ and $JX_n^{-}$ as above. If $n$ is even, then $JX_n^{-}$ is trivial. If $n$ is odd, the dimension of $JX_n^{-}$ is  $\frac{n-1}{2}$.
\end{lemma}

\begin{proof}
Since we have the action of $E_n\cong (\mathbb{Z}/2\Z)^n$ on $X_n$, we know \cite{lr} that $JX_n$ decomposes, in general, as
\[JX_n \sim \bigoplus_{\rho \in \irr_{\mathbb{Q}}(E_n)} B_{\rho},\]
where $\irr_{\mathbb{Q}}(E_n)$ is the set of rational irreducible representations of $E_n$, up to equivalence. Moreover, each subvariety $B_{\rho}$ corresponds to the image of the central idempotent $e_{\rho}$ defined by ${\rho}$ in the group algebra $\mathbb{Q}[E_n]$.

In our case, we know from \cite[Thm. 3.4]{Alexis} and the Riemann-Hurwitz formula that the action of $E_n$ on $X_n$ is of signature $(0;2^{n+1})$. That is, the total quotient $X_n/E_n$ is of genus $0$, and there are $n+1$ orbits of fixed points, each with a stabilizer of order $2$. In fact, the stabilizers are the $\langle\sigma_i\rangle$ for $i=0,\dots, n$. Furthermore, the subvariety $JX_n^{-}$ corresponds to the representation $\rho$ of $E_n$ given by $\rho(\sigma_i)=-1$ for $i=0,\dots, n$. Notice that since $\prod_i \sigma_i=1$, $\rho$ is defined only in the case of an odd $n$. It follows immediately that for an even $n$, the dimension of $JX_n^{-}$  is $0$ (see also \cite[Remark 3.6]{Alexis}).

To find the dimension of $JX_n^{-}$ for odd $n$, we apply the formula in \cite[Thm. 5.12]{ro} (see also \cite{ksir}). Following the notations therein, we have $k_{\rho}=1$, $\dim \rho = 1$, $\gamma=0$, $t=n+1$, and $\dim \fix_{G_k}\rho=0$ since $G_k=\langle \sigma_k\rangle$. This gives immediately $\dim B_{\rho}=\frac{n-1}{2}$.
\end{proof}

We are ready now to prove Theorem \ref{main thm}.

\begin{proof}[Proof of Theorem \ref{main thm}]
The statement on the dimension of $JX_n^-$ is Lemma \ref{Anitaslemma}. The proof of the decomposition follows easily by induction: For $n=3$, we have that $X_n$ is an elliptic curve and hence the decomposition \eqref{eq decomp alexis} gives us $JX_3=JX_3^{-}$ since a Humbert-Edge curve of type 2 is simply a quadric in $\P^2$, hence of genus 0. This gives the result in this case.

Assume now that the result holds for $n-1$ and let us prove it for $n$. Using once again the decomposition \eqref{eq decomp alexis} for $X_n$ we get
\[JX_n\sim JX_n^-\oplus\sum_{i=0}^n\pi_i^* J(X/\sigma_i).\]
Now, using the induction hypothesis for $n-1$ and the fact that the projection $\pi_T:X_n\to X_T$ factors naturally through $X/\sigma_i$ if $\sigma_i\in T$, we get
\[JX_n\sim JX_n^{-} \oplus\sum_{i=0}^n\underset{|T|\leq n-3}{\bigoplus_{\sigma_i\in T\subset S_n}}\pi_T^*JX_T^{-}.\]
Note that $\pi_T^*JX_T^-$ appears in the sum for \emph{every non-empty} $T$ with $|T|\leq n-3$. Considering then, for $i\neq j$ and $T\supset\{\sigma_i,\sigma_j\}$, the commutative diagrams
\[\xymatrix@=.5em{
& X_n \ar[dl]_{\pi_i} \ar[dr]^{\pi_j} \ar[dd]^{\pi_{T}} & \\
X/\sigma_i \ar[dr] & & X/\sigma_j \ar[dl] \\
& X_{T}
}\qquad
\xymatrix@=.5em{
& JX_n & \\
X/\sigma_i \ar[ur]^{\pi_i^*} & & X/\sigma_j \ar[ul]_{\pi_j^*} \\
& JX_{T} \ar[ul] \ar[ur] \ar[uu]_{\pi_{T}^*}
}\]
we see that the sum from $i=0$ to $n$ is simply making up $|T|$ copies of each $\pi_T^*JX_T^-$, which we may clearly erase then in order to get
\[JX_n\sim JX_n^{-} \oplus\underset{|T|\leq n-3}{\bigoplus_{\emptyset\neq T\subset S_n}}\pi_T^*JX_T^{-},\]
which is the desired decomposition since $JX_n^-=JX_T^-$ for $T=\emptyset$.
\end{proof}

\section{Geometric description of the decomposition of $JX_n$}
In this section, we describe the kernel of the decomposition displayed in Theorem \ref{main thm} and the induced polarization on each factor. We prove that each factor in this decomposition for $JX_n$ is a Prym-Tyurin variety of exponent $2^{n-3}$ with respect to $X_n$. This description of the induced polarization on each factor allows us to compute the kernel of the isogeny, as we show in Lemma \ref{L:kernel}. The main results in this section are collected in the following theorem:

\begin{theorem}\label{T:PT-K}
Let $X_n$ be a Humbert-Edge curve of type $n\geq 3$ and genus $g_n$ decomposed as
\[\varphi: \underset{|T|\leq n-3}{\bigoplus_{T\subset S_n}}\pi_T^*JX_{T}^{-} \to JX_n\]
Then
\begin{enumerate}
\item Each factor is a Prym-Tyurin variety of exponent $2^{n-3}$ with respect to $X_n$.

\item The kernel of the isogeny $\varphi$ is of order $(2^{n-3})^{g_n}$.
\end{enumerate}
\end{theorem}

We prove this theorem in several steps, keeping the notations we have defined in the previous sections. Let us start by the following remark concerning $JX_n^-$:

\begin{lemma}\label{lemma H}
Let $X_n$ be a Humbert-Edge curve of type $n$. Let $H_n\leq E_n$ be the subgroup generated by the double products $\sigma_i\sigma_j$ for $i,j\in\{0,\ldots,n\}$. Then $JX_n^-=\pi_{H_n}^*J(X_n/H_n)$, where $\pi_{H_n}$ denotes the canonical projection $X_n\to X_n/H_n$.
\end{lemma}

\begin{proof}
Recall \cite[Prop. 5.2]{cr}  that $\pi_{H_n}^*J(X_n/H_n)$ is the subvariety defined as the image of the idempotent $p_{H_n}=\frac{1}{|H|}\sum_{h \in H_n}h$, which is the  neutral connected component of $(JX_n)^{H_n}$. The result then follows immediately from \cite[Prop.~3.5(ii)]{Alexis}. Here once again there is an issue with connectedness, but it does not affect our result since we are taking the neutral connected component on both sides of the equality.
\end{proof}

Using Lemma \ref{Anitaslemma} and the Riemann-Hurwitz formula, we immediately get the following result:

\begin{corollary}\label{cor H}
Let $X_n$ and $H_n$ be as above and assume that $n$ is odd. Then the quotient morphism $\pi_H:X_n\to X_n/H_n$ is \'etale.
\end{corollary}

\begin{rem}
Lemma \ref{lemma H} may sound troubling in the case of even $n$ since then $JX_n^-$ is trivial. However, note that in this case $H_n$ is the whole $E_n$ and thus $X_n/H_n=X_n/E_n$. And as we saw in the proof of Lemma \ref{Anitaslemma}, $X_n/E_n$ has genus $0$, hence its Jacobian is trivial.

In the odd case, it is easy to see that $H_n$ has index 2 in $E_n$, hence $|H_n|=2^{n-1}$.
\end{rem}

\begin{rem}
Note that Lemma \ref{lemma H} and Theorem \ref{T:PT-K} imply that the Jacobian of a Humbert-Edge curve decomposes as a product of Jacobians. 
\end{rem}

We continue now with a lemma on abelian \'etale covers of curves, such as the one in the last corollary.

\begin{lemma}\label{lem ker}
Let $f:X\to Y$ be a Galois morphism between smooth projective curves with Galois group $G$ and let $f^*:JY\to JX$ be the corresponding morphism on the Jacobian varieties. Let $g$ be the dimension of $Y$ and assume that $f$ is of degree $p^{2g}$ for some prime $p$. Then the following are equivalent:
\begin{enumerate}
\item $f$ is \'etale and $G$ is isomorphic to $(\Z/p\Z)^{2g}$.
\item $f$ factors through every cyclic \'etale cover $Z\to Y$ of degree $p$.
\item $\ker f^*=JY[p]$.
\end{enumerate}
\end{lemma}

\begin{proof}
From \cite[Prop.~11.4.3]{bl} and its proof, one easily sees that a non-trivial element $\xi\in JY[p]$ is in $\ker f^*$ if and only if $f$ factors through a certain cyclic \'etale cover $X_\xi\to Y$ of degree $p$ that is uniquely defined by $\langle\xi\rangle$ up to $Y$-isomorphism. Moreover, cyclic subgroups of $JY[p]$ classify all cyclic \'etale covers of degree $p$ in this way (see the comment right after \textit{loc.~cit.}). Then (3) implies (2) and (2) implies $\ker f^*\supset JY[p]$. However, from \textit{loc.~cit.} one deduces that $|\ker f^*|$ is smaller or equal than the degree of $f$, which is $p^{2g}$, hence (3).

Assume (1). Then there are precisely $(p^{2g}-1)/(p-1)$ non-trivial quotients of $G$ isomorphic to $\Z/p\Z$. It is easy to see that each one of these induces a different cyclic \'etale cover of degree $p$ (up to $Y$-isomorphism). Then $f$ factors through all of these covers and we have (2) since this is exactly the number of cyclic subgroups of $JY[p]$.

Finally, assume (2). As we just said, there are $(p^{2g}-1)/(p-1)$ different cyclic \'etale cover of degree $p$ and $f$ factors through all of them. Hence, there are at least $(p^{2g}-1)/(p-1)$ non-trivial quotients of $G$ isomorphic to $\Z/p\Z$. Now, the only group of order $p^{2g}$ having this property is $(\Z/p\Z)^{2g}$. It is easy then to see that $f$ must be \'etale: Take for instance the fiber product of $2g$ cyclic covers coming from a generating set of $G=(\Z/p\Z)^{2g}$. Then $X$ is isomorphic to this product, which is \'etale over $Y$.
\end{proof}

These results allow us to understand the polarization induced on $JX_n^-$. Indeed:

\begin{proposition}\label{prop JX- is PT}
Let $X_n$ be a Humbert-Edge curve of type $n$ for odd $n$ and let $\Theta$ be the principal polarization of $JX_n$. Then the polarization induced by $\Theta$ on $JX_n^-$ is of type $(2^{n-3},\ldots,2^{n-3})$.
\end{proposition}

\begin{proof}
Using Lemma \ref{lemma H}, we see $JX_n^-$ as $\pi_{H_n}^*J(X_n/H_n)$. Then, since $|H_n|=2^{n-1}$, we know by \cite[Lem.~12.3.1]{bl} that the polarization $(\pi_{H_n}^*)^*\Theta$ is $2^{n-1}$ times the principal polarization of $J(X_n/H_n)$. By Lemma \ref{lem ker} we have $\ker(\pi_{H_n}^*)=J(X_n/H_n)[2]$ and hence we have the following commutative diagram
\[\xymatrix{JX_n^- \ar@{^{(}->}[r]^{i} & JX_n \\ & J(X_n/H_n) \ar[u]_{\pi_{H_n}^*} \ar@{->>}[ul]^{\times 2}}\]
This immediately gives
\[2^{n-1}\Theta_{X_n/H_n}\equiv\pi_{H_n}^*\Theta\equiv 2^*i^*\Theta\equiv4i^*\Theta\]
and since the N\'eron-Severi group of an abelian variety is torsion-free, we have that $i^*\Theta\equiv 2^{n-3}\Theta_{X_n/H_n}$, and the result follows.
\end{proof}

\begin{rem}
From this proof one easily deduces that under the hypothesis (3) of Lemma \ref{lem ker}, even without assumptions on the order of $f$ or whether it is Galois, the subvariety $f^*JY\subset JX$ is a Prym-Tyurin variety of exponent $p^{2g-2}$.
\end{rem}

We are now ready to prove statement (1) in Theorem \ref{T:PT-K}.

\begin{proposition}
Let $n\geq 3$ be an integer, $X_n$ a Humbert-Edge curve of type $n$ and $T\subset S_n$ with $|S_n\smallsetminus T|>3$. Then, the induced polarization on the factor $\pi_T^* JX_T^-$ is of type $(2^{n-3},\ldots, 2^{n-3})$.
\end{proposition}

\begin{proof}
This is easily proved by induction using Proposition \ref{prop JX- is PT}. For $n=3$, $JX_3=JX_3^-$ and it is an elliptic curve with polarization of type $(1)=(2^{3-3})$.

For $n\geq 4$, assume the result for $n-1$. If $T=\emptyset$, then the result corresponds to Proposition \ref{prop JX- is PT} when $n$ is odd and $JX_T^-$ is trivial when $n$ is even, so we need not consider this case. If $T\neq\emptyset$, consider an element $\sigma_i\in T$. Then $\pi_T$ factors as
\[X_n\xrightarrow{\pi_i} X_n/\sigma_i\xrightarrow{\pi'} X_T,\]
and hence $\pi_T^*$ factors as
\[JX_T\xrightarrow{{\pi'}^*} J(X_n/\sigma_i)\xrightarrow{\pi_i^*} JX_n.\]
By the induction hypothesis, ${\pi'}^*JX_T^-$ is a Prym-Tyurin subvariety of $J(X_n/\sigma_i)$ of exponent $2^{n-4}$. Recalling then that $\pi_i$ is of degree $2$ and ramified, it follows from \cite[Prop. 12.3.1]{bl} and \cite[Prop. 11.4.3]{bl} that the polarization on $\pi_T^*JX_n$ is of type $(2^{n-3},\ldots,2^{n-3})$.
\end{proof}

In order to prove statement (2) in Theorem \ref{T:PT-K}, we prove the following lemma. The result follows immediately then from statement (1).

\begin{lemma}\label{L:kernel}
Let $(A,\Theta)$ be a principally polarized abelian variety. Assume that there exist abelian subvarieties $A_1,\ldots, A_r$ of $A$ such that $i_j^*\Theta\equiv d\theta_j$ for a principal polarization $\theta_j$ on $A_j$, for $j=1,\ldots,r$. Denote by $\varphi:\bigoplus_{j=1}^r A_i\to A$ the isogeny given by the sum. Then 
\[|\ker \varphi|=d^{\dim A}.\]
\end{lemma}

\begin{proof}
We know that $\varphi^*\Theta=d(\theta_1\boxtimes \cdots \boxtimes \theta_r)$, hence
\[(\varphi^*\Theta)^{\dim A}=d^{\dim A}(\theta_1\boxtimes \cdots \boxtimes \theta_r)^{\dim A}=d^{\dim A}(\dim A)!,\]
where the last equality follows from Riemann-Roch for abelian varieties. On the other hand,
\[(\varphi^*\Theta)^{\dim A}=\deg \varphi \cdot \Theta^{\dim A}.\]
Once again by Riemann-Roch, we have $\Theta^{\dim A}=(\dim A)!$ and the result follows.
\end{proof}

\begin{rem}
Using statement (1) in Theorem \ref{T:PT-K} and \cite[Prop.~12.1.9]{bl}, we finally deduce that $JX_n^{-}$, as defined in \cite{Alexis}, is connected.
\end{rem}

\bibliographystyle{amsplain}

\end{document}